\documentclass{amsart}

\usepackage[utf8]{inputenc}
\usepackage{color}
\usepackage{xcolor}
\usepackage{hyperref}
\usepackage{amssymb}
\usepackage{float}
\usepackage{nicefrac}
\usepackage{enumerate}
\usepackage[nameinlink, capitalize, noabbrev]{cleveref}

\hypersetup{
    colorlinks=true,
    linkcolor=teal,
    citecolor=magenta,
    }

\newtheorem{Theorem}{Theorem}

\newtheorem{proposition}{Proposition}[section]
\newtheorem{lemma}[proposition]{Lemma}
\newtheorem{corollary}[proposition]{Corollary}
\newtheorem{theorem}[proposition]{Theorem}

\theoremstyle{definition}

\newtheorem{definition}[proposition]{Definition}

\numberwithin{equation}{section}

\title{Branch actions and the structure lattice}
\author{Jorge Fariña-Asategui and Rostislav Grigorchuk}
\address{Jorge Fariña-Asategui: Centre for Mathematical Sciences, Lund University, 223 62 Lund, Sweden -- Department of Mathematics, University of the Basque Country UPV/EHU, 48080 Bilbao, Spain}
\email{jorge.farina\_asategui@math.lu.se}
\address{Rostislav Grigorchuk: Department of Mathematics, Texas A\&M University, 77843 College Station, U.S.A.}
\email{grigorch@math.tamu.edu}
\thanks{The first author is supported by the Spanish Government, grant PID2020-117281GB-I00, partly with FEDER funds. The first author also acknowledges support from the Walter Gyllenberg Foundation from the Royal Physiographic Society of Lund. The second  author  is  supported by Travel
Support for Mathematicians, grant MP-TSM-00002045 from Simons Foundation.}

\begin{document}

\begin{abstract}

J.\,S. Wilson  proved in 1971 an isomorphism  between  the  structural  lattice  associated to  a  group belonging  to his  second  class  of  groups  with  every  proper  quotient  finite and the  Boolean  algebra  of  clopen  subsets  of  Cantor's ternary set. In this paper we generalize this isomorphism to the  class  of  branch  groups. Moreover, we show that for every faithful  branch action of a group $G$ on a spherically homogeneous rooted tree $T$ there  is  a canonical $G$-equivariant isomorphism  between the Boolean algebra associated  to  the  structure lattice of $G$ and the Boolean algebra of clopen   subsets  of the boundary of $T$. 

\end{abstract}

\maketitle

\section{introduction}
\label{section: introduction}

Branch groups are groups acting level-transitively on spherically homogeneous rooted trees whose subnormal subgroup structure resembles the one of the full automorphism group of the tree. The class of branch groups was introduced by the second author in 1997, as a common generalization of the different examples of Burnside groups and groups of intermediate growth introduced in the 1980s by the second author, Gupta and Sidki, and Suschansky among others; compare \cite{GrigorchukBurnside, GrigorchukMilnor, GuptaSidki, Suschansky}.

Just-infinite branch groups constitute one of the three classes of just-infinite groups, i.e. those infinite groups whose proper quotients are all finite; see \cite{GrigorchukNewHorizons}. This classification of just-infinite groups by the second author in \cite{GrigorchukNewHorizons} relies on Wilson's classification in terms of their structure lattice in \cite{Wilson71}. A closely related object to the structure lattice of a branch group $G$ is the structure graph, which is the subgraph of the structure lattice consisting of the basal subgroups of $G$. The structure graph of a branch group was introduced by Wilson in~\cite{NewHorizonsWilson}.

In the last decades, the study of the structure graph of a branch group has attracted a lot of attention, as the structure graph encodes all the different branch actions of the group on a spherically homogeneous rooted tree; compare \cite{AlejandraCSP,  GrigorchukWilson, Hardy, NewHorizonsWilson, WilsonBook}. For instance, Hardy proved in \cite[Theorem 15.4.2]{Hardy} (see also \cite[Theorem~5.2]{WilsonBook}) that a group admits a unique branch action on the $p$-adic tree for a prime $p\ge 2$ if and only if its structure graph is isomorphic as a graph to the $p$-adic tree. Sufficient conditions for the uniqueness of branch actions on the $p$-adic tree were given by the second author and Wilson in \cite{GrigorchukWilson}. A complete description of the structure graph of a branch group was given by Hardy in his PhD thesis \cite[Lemma 15.1.1]{Hardy}; see also \cite[Lemma 5.1(a)]{WilsonBook}.

In contrast, there is no explicit description of the structure lattice of a branch group. Indeed, to the best of our knowledge, the only known structural result on the structure lattice of a branch group is the following theorem proved by Wilson in 1971:

\begin{theorem}[{see {\cite[Theorem 6]{Wilson71}}}]
\label{theorem: Wilson Thm 6}
    Let $G\le \mathrm{Aut}~T$ be a just-infinite branch group. Then its structure lattice $\mathcal{L}(G)$ is isomorphic to the lattice of clopen subsets of Cantor's ternary set.
\end{theorem}

The isomorphism in \cref{theorem: Wilson Thm 6} is given by Stone's representation theorem and the homeomorphism in \cite[Theorem 2.97]{Hocking} between the Stone space of the structure lattice and Cantor's ternary set. However, the isomorphism in \cref{theorem: Wilson Thm 6} does neither preserve the branch action of the group $G$ nor describe explicitly the equivalence classes in the structure lattice $\mathcal{L}(G)$.

Our goal is to give a complete description of the structure lattice of a branch group. Indeed, we obtain a $G$-equivariant and explicit canonical isomorphism of Boolean algebras characterising the structure lattice of a (non-necessarily just-infinite) branch group $G$:

\begin{Theorem}
\label{Theorem: canonical isomorphism of boolean algebras}
    Let $\rho:G\to \mathrm{Aut}~T$ be a branch action of a group $G$ on a spherically homogeneous rooted tree $T$ and $\mathrm{Bool}(\partial T)$ the Boolean algebra of clopen subsets of the boundary of $T$. Then there exists a canonical $G$-equivariant isomorphism $\Phi:\mathcal{L}(G)\to \mathrm{Bool}(\partial T)$ with inverse $\Phi^{-1}:\mathrm{Bool}(\partial T)\to \mathcal{L}(G)$ given by
    \begin{align*}
    \Phi:\mathcal{L}(G)&\to \mathrm{Bool}(\partial T)\quad \quad\text{and}&\Phi^{-1}:\mathrm{Bool}(\partial T)&\to \mathcal{L}(G)\\
        [H]&\mapsto \mathrm{Supp}(H),&C&\mapsto \Big[\prod_{v\in C} \mathrm{rist}_{\rho(G)}(v)\Big],
    \end{align*}
    where $\mathrm{Supp}(H)$ is the support of $H$, i.e. the elements in $\partial T$ not fixed by $\rho(H)$.
\end{Theorem}

\cref{Theorem: canonical isomorphism of boolean algebras} shows that for a branch group $G$ every $G$-action on $\partial T$ induced by a branch action $\rho:G\to\mathrm{Aut}~T$ is equivalent to the action of $G$ by conjugation on its structure lattice $\mathcal{L}(G)$. In particular the action of $G$ on the Stone space of its structure lattice is equivalent to the action of $G$ on $\partial T$. \cref{Theorem: canonical isomorphism of boolean algebras} also gives an explicit description of the structure lattice analogous to the one of the structure graph by Hardy in  \cite[Lemma 15.1.1]{Hardy}.

\subsection*{\textit{\textmd{Organization}}}  In \cref{section: Boolean algebras and Stone spaces} we introduce Boolean algebras and Stone spaces. The special case of the Boolean algebra of clopen subsets of the boundary of a spherically homogeneous rooted tree is treated in \cref{section: boundary of a rooted tree}. Finally, branch groups and their structure lattice are defined in \cref{section: the structure lattice of a branch group} and we conclude the section with a proof of \cref{Theorem: canonical isomorphism of boolean algebras}.

\subsection*{\textit{\textmd{Notation}}} We use the exponential notation for the right action of a group on a tree and on its boundary. We write $H\le_f G$ if $H$ is a finite-index subgroup of $G$.

\subsection*{Acknowledgements} The first author would like to thank Texas A\&M University for its warm hospitality while this work was being carried out.

\section{Boolean algebras and Stone spaces}
\label{section: Boolean algebras and Stone spaces}

In this section, we introduce the main concepts and results on Boolean algebras and Stone spaces which shall be needed later on in the paper. We follow mainly~\cite{Halmos}.

\subsection{Boolean algebras}

Here we introduce the main concepts needed in order to define Boolean algebras. 

\begin{definition}[Poset]
    Let $A$ be a set and $\le$ a relation satisfying the following properties:
    \begin{enumerate}[\normalfont(i)]
        \item \textit{reflexive}: for any $a\in A$ we have $a\le a$;
        \item \textit{transitive}: for any $a,b,c\in A$ if $a\le b$ and $b\le c$ then $a\le c$;
        \item \textit{anti-symmetric}: for any $a,b\in A$ if $a\le b$ and $b\le a$ then $a=b$.
    \end{enumerate}
    Then the pair $(A,\le)$ is said to be a \textit{partially-ordered set} or \textit{poset}.
\end{definition}


\begin{definition}[Meet and join]
    Let $(A,\le)$ be a poset. We define  the \textit{meet} of $F\subseteq A$ as the unique element $\bigwedge F$ (if it exists) such that
\begin{enumerate}[\normalfont(i)]
    \item for any $f\in F$ we have $\bigwedge F\le f$;
    \item for any $a\in A$ if $a\le f$ for all $f\in F$ then $a \le \bigwedge F$.
\end{enumerate}
 
Similarly, we define the \textit{join} of $F$ as the unique element $\bigvee F$ (if it exists) such that
\begin{enumerate}[\normalfont(i)]
    \item for any $f\in F$ we have $f\le \bigvee F$;
    \item for any $a\in A$ if $f\le a$ for all $f\in F$ then $\bigvee F\le a$.
\end{enumerate}
\end{definition}

For single elements $a,b\in A$ we shall denote their meet and join by $a\wedge b$ and $a\vee b$ respectively.

\begin{definition}[Distributive lattice]
Let $(A,\le)$ be a poset. If for any finite subset $F\le A$ both the meet and the join of $F$ exist in $A$ then we say that $(A,\le)$ is a \textit{lattice}. We say a lattice $(A,\le)$ is \textit{distributive} if for any $a,b,c\in A$ we have
$$a\wedge (b\vee c)=(a\wedge b)\vee (a\wedge c).$$
\end{definition}

The simplest example of a distributive lattice is the power set of a finite set ordered by inclusion. Here the meet and the join correspond to the intersection and the union of subsets respectively. In particular, for a singleton $\{a\}$, its power set has simply two elements, namely $\emptyset$ and $\{a\}$. We shall write $0:=\emptyset $ and $1:=\{a\}$ and denote by $\mathbf{2}$ the distributive lattice $(\{0,1\},\le )$, where $0\le 1$. Note that
$$0\vee 0=0\wedge 0=0\wedge 1=1\wedge 0=0, \quad 1\vee 1=1\wedge 1=1\vee 0=0\vee 1=1.$$
More generally in a distributive lattice $(A,\le)$, we define the distinguished elements $0,1\in A$ as the unique elements (if they exist) such that
$$a\wedge 0=0, \quad a\wedge 1=a, \quad a\vee 0=a,\quad \text{and}\quad a\vee 1=1$$
for every $a\in A$. Note that for the power set of a finite set $S$ we have $0=\emptyset$ and $1=S$.

\begin{definition}[Complement]
Let $(A,\le)$ be a distributive lattice admitting the distinguished elements 0 and 1. Then for $a\in A$ we define its \textit{complement} (if it exists) as the unique element $\neg a$ satisfying the following:
\begin{enumerate}[\normalfont(i)]
    \item $a\wedge \neg a=0$;
    \item $a\vee \neg a=1$.
\end{enumerate}
\end{definition}

Note that by definition $\neg 0=1$ and $\neg 1=0$ and recall that for every $a,b\in A$ we have de Morgan's laws: $\neg(a\vee b)=\neg a\wedge \neg b$ and $\neg(a\wedge b)=\neg a\vee \neg b$.

\begin{definition}[Boolean algebra]
    Let $(A,\le)$ be a distributive lattice admitting the distinguished elements 0 and 1. If every element $a\in A$ admits a complement we say that $(A,\le)$ is a \textit{Boolean algebra}.
\end{definition}

The lattice $\mathbf{2}$, and more generally the power set of a finite set, are examples of Boolean algebras. Another important example of Boolean algebras comes from topology. Given a topological space $X$, the set of clopen subsets of $X$ forms a Boolean algebra with meet and join given by intersection and union respectively and the distinguished elements $0=\emptyset$ and $1=X$.

\subsection{Homomorphisms of Boolean algebras}

In order to link Boolean algebras to Stone spaces we need the concept of homomorphisms of Boolean algebras:

\begin{definition}[Homomorphism of Boolean algebras]
Let $A$ and $B$ be two Boolean algebras. We say a map $f:A\to B$ is a \textit{homomorphism of Boolean algebras} if it preserves meets, joins, and the distinguished elements $0,1$ respectively. In other words
$$f(a\wedge b)=f(a)\wedge f(b),\quad f(a\vee b)=f(a)\vee f(b),\quad f(0)=0,\quad \text{ and }\quad f(1)=1$$
for every $a,b\in A$.
\end{definition}

A direct consequence of the above definition is that a homomorphism $f:A\to B$ preserves complements, i.e. $f(\neg a)=\neg f(a)$ for every $a\in A$.

\begin{definition}[Ideals and maximal ideals]
    Let $A$ be a Boolean algebra. A non-empty subset $I\subseteq A$ is called an \textit{ideal} if it satisfies the following properties:
    \begin{enumerate}[\normalfont(i)]
        \item for all $a,b\in I$, we have $a\vee b\in I$;
        \item for all $a\in A$ and $b\in I$, we have $a\wedge b\in I$.
    \end{enumerate}

    Furthermore, a proper ideal $I\subset A$ is said to be \textit{maximal} if for any proper ideal $J\subset A$ the inclusion $I\subseteq J$ implies $I=J$.
\end{definition}

A first example of an ideal in a Boolean algebra $A$ is the kernel of any Boolean algebra homomorphism $f:A\to B$. Indeed, every ideal in $A$ arises as the kernel of a Boolean algebra homomorphism $f:A\to B$ for some $B$; see \cite[Chapter 18]{Halmos}.

The following lemma \cite[Chapter 20, Lemma 1]{Halmos} characterizes when an ideal is maximal.

\begin{lemma}[{see \cite[Chapter 20, Lemma 1]{Halmos}}]
\label{lemma: maximal ideals contain either a or its complement}
    Let $A$ be a Boolean algebra and $I\subseteq A$ an ideal. Then $I$ is maximal if and only if for every $a\in A$ either $a\in I$ or $\neg a\in I$ but not both.
\end{lemma}

We can use \cref{lemma: maximal ideals contain either a or its complement} to characterize when a map $f:A\to\mathbf{2}$ is a homomorphism of Boolean algebras.

\begin{lemma}
\label{lemma: ideal sufficient}
    Let $A$ be a Boolean algebra and $f:A\to \mathbf{2}$ a map. Then $f$ is an homomorphism of Boolean algebras if and only if the set $\{a\in A \mid f(a)=0\}$ is a maximal ideal in $A$.
\end{lemma}
\begin{proof}
    First note that if $f$ is a homomorphism of boolean algebras then the set $I:=\{a\in A \mid f(a)=0\}$ is an ideal of $A$ as it is the kernel of the homomorphism~$f$. Furthermore it is maximal by \cref{lemma: maximal ideals contain either a or its complement} as for every $a\in A$ either $f(a)=0$ or $f(\neg a)=\neg f(a)=\neg 1=0$ but not both $f(a)=f(\neg a)=0$, as otherwise $f(1)=f(a\vee \neg a)=f(a)\vee f(\neg a)=0$. Now let us assume that the set $I$ is a maximal ideal of~$A$ and let us prove that $f$ is then a homomorphism. We only need to show that if $a,b\notin I$ then $a\wedge b\notin I$ and that for every $a\in A$ if $b\notin I$ then $a\vee b\notin I$, as the other conditions follow directly from the axioms on the definition of an ideal and $I$ being maximal. If $a,b\notin I$ then $\neg (a\wedge b)=\neg a\vee \neg b\in I$ by \cref{lemma: maximal ideals contain either a or its complement} and thus $a\wedge b\notin I$ again by \cref{lemma: maximal ideals contain either a or its complement}. Lastly if $a\in A$ and $b\notin I$ then $\neg (a\vee b)=\neg a \wedge \neg b\in I$ by \cref{lemma: maximal ideals contain either a or its complement} and hence $a\vee b\notin I$ also by \cref{lemma: maximal ideals contain either a or its complement}.
\end{proof}

\subsection{Stone spaces}

We conclude the section by defining Stone spaces and stating some of their useful properties.

\begin{definition}[Stone space]
Let $X$ be a topological space. We say that $X$ is a \textit{Stone space} if it is Hausdorff, compact and totally disconnected. 
\end{definition}

A Stone space is also widely known as a profinite space; see \cite{RibesZalesskii}. The following lemma tells us how the clopen subsets of a Stone space look like:

\begin{lemma}
\label{lemma: disjoint clopens}
    Let $X$ be a Stone space and let $\mathcal{U}$ be a basis of opens for the topology in $X$. Then any clopen subset $C\subseteq X$ is a finite union of opens in $\mathcal{U}$. Furthermore, if $\mathcal{U}$ consists of clopen subsets of $X$ such that for every pair $U_1,U_2\in \mathcal{U}$ either $U_1\subseteq U_2$, $U_2\subseteq U_1$ or $U_1\cap U_2=\emptyset$, then $C$ may be represented as the finite disjoint union of clopen subsets in $\mathcal{U}$.
\end{lemma}
\begin{proof}
    Since $C$ is open it is an arbitrary union of open subsets in $\mathcal{U}$, i.e. $C=\bigcup_{i\in I}U_i$. Since $C$ is closed in $X$ and $X$ is a Stone space $C$ is compact, thus there exists a finite refinement $C=\bigcup_{i=1}^n U_i$. To obtain a disjoint union, we remove unnecessary terms from the finite decomposition $C=\bigcup_{i=1}^n U_i$ using the fact that for $1\le j<k\le n$, the intersection $U_j\cap U_k$ is clopen and equal to either $U_j,U_k$ or~$\emptyset$.
\end{proof}

An important example of a Stone space is the Stone space associated to a Boolean algebra. Given a Boolean algebra $A$ its \textit{Stone space} is simply $\mathrm{Hom}(A)$, i.e. the space consisting of all Boolean algebra homomorphisms $A\to \mathbf{2}$, which is a Stone space for the topology of pointwise convergence of nets. This topology is no more than the Tychonoff topology in $\mathrm{Hom}(A)$ when seen as the cartesian product $\mathbf{2}^A$.

\begin{lemma}
\label{lemma: non zero map}
    Let $X$ be a topological space such that $\mathrm{Bool}(X)\ne \mathbf{2}$. Then for every homomorphism of Boolean algebras $f:\mathrm{Bool}(X)\to \mathbf{2}$ there exists some non-empty proper clopen $C$ such that $f(C)=1$.
\end{lemma}
\begin{proof}
    Since $\mathrm{Bool}(X)\ne \mathbf{2}$ there must exist a proper clopen subset $\emptyset \ne U\subset X$. Thus, its complement $\neg U$ is also a non-empty proper clopen subset of $X$. Then we have $$1=f(X)=f(U\vee \neg U)=f(U)\vee f(\neg U),$$
    which implies that either $f(U)=1$ or $f(\neg U)=1$.
\end{proof}

We conclude the section with a folklore result on Cantor sets, which seems to not be recorded anywhere:

\begin{proposition}
    Let $X$ be a Cantor set. Then the automorphism group of the Boolean algebra  $\mathrm{Bool}(X)$ is isomorphic to the group of homeomorphisms of $X$.
\end{proposition}
\begin{proof}
    Let $f\in \mathrm{Homeo}~X$. Then for any clopen subset $C\subseteq X$ the image $f(C)\subseteq X$ is again clopen. Since $f(X)=X$ and $f(\emptyset)=\emptyset$ and for any pair of clopens $C_1\subseteq C_2\subseteq X$ we have $f(C_1)\subseteq f(C_2)\subseteq X$, the homeomorphism $f$ induces an automorphism $f'$ of the Boolean algebra $\mathrm{Bool}(X)$. Now note that a Cantor set~$X$ is totally separated, i.e. any point $x\in X$ is given by the intersection $x=\bigcap C$, where $C$ ranges over all the of clopen subsets containing $x$. Thus, an automorphism $f\in \mathrm{Aut}(\mathrm{Bool}(X))$ induces a map $\widetilde{f}:X\to X$ given by
    $$\widetilde{f}(x):=\bigcap f(C).$$
    The map $\widetilde{f}$ is bijective and continuous (with continuous inverse) by definition, thus $\widetilde{f}\in \mathrm{Homeo}~X$. The map $F':\mathrm{Homeo}~X\to \mathrm{Aut}(\mathrm{Bool}(X))$ given by $f\mapsto f'$ is a group homomorphism and similarly the map $\widetilde{F}: \mathrm{Aut}(\mathrm{Bool}(X))\to \mathrm{Homeo}~X$ given by $f
    \mapsto \widetilde{f}$ is a group homomorphism. Finally $\widetilde{F}\circ F'=\mathrm{id}_{\mathrm{Homeo~}X}$ and $F'\circ \widetilde{F}=\mathrm{id}_{\mathrm{Aut(Bool}(X))}$ as
    $$\widetilde{(f')}(x)=\bigcap f(C)=f\big(\bigcap C\big)=f(x)$$
    for any $x\in X$ and $f\in \mathrm{Homeo}~X$, and
    $$(\widetilde{g})'(C)=g(C)$$
    for any clopen subset $C\subseteq X$ and $g\in \mathrm{Aut}(\mathrm{Bool}(X))$. Hence $\mathrm{Homeo}~X$ is isomorphic to $\mathrm{Aut}(\mathrm{Bool}(X))$.
\end{proof}

\section{The boundary of a spherically homogeneous rooted tree}
\label{section: boundary of a rooted tree}

In this section we introduce spherically homogeneous rooted trees and their boundaries and study the Boolean algebra of clopens of the latter.

\subsection{Spherically homogeneous rooted trees and their boundaries}

A \textit{spherically homogeneous rooted tree} $T$ is a rooted tree such that every vertex in $T$ at the same distance from the root has the same degree. For every $n\ge 1$, the vertices at distance $n$ from the root constitute the $n$th \textit{level} of $T$, denoted $\mathcal{L}_n$. For a vertex $v\in T$, the subtree rooted at $v$, denoted $T_v$, consists of the vertices in $T$ below $v$.

The tree $T$ can be identified with the set of finite words $\prod_{n\ge 1} X_n$, where the cardinality of each set $X_n$ coincides with the degree of the vertices at level $n-1$ for $n\ge 1$. Note that under this identification the root is simply the empty word. Also this identification induces a graded lexicographical order in $T$, by fixing a lexicographical order in $X_n$ for every $n\ge 1$.

Two vertices $u,v\in T$ are said to be \textit{incomparable} if none is a descendant of the other.

Let $T$ be a spherically homogeneous rooted tree. We define its boundary $\partial T$ as the set of ends in $T$. In other words, an element of the boundary $\gamma\in \partial T$ is simply an infinite path in $T$. For each vertex $v\in T$ we define the \textit{cone set} $C_v\subseteq \partial T$ as the subset of all paths in $\partial T$ passing through $v$.  We write $v\in \gamma$ if the path $\gamma$ passes through $v$. Note that
$$\gamma=\bigcap_{v\in \gamma} C_v.$$

The boundary $\partial T$ is a Stone space with respect to the compact topology generated by the cone sets $\{C_v\}_{v\in T}$. Cone sets are clopen sets in this topology and they satisfy all the assumptions in \cref{lemma: disjoint clopens}. Thus any clopen in $\partial T$ may be represented as a finite disjoint union of cone sets. For a vertex $v\in T$ and a clopen $C\subseteq \partial T$ we write $v\in C$ if $C_v\subseteq C$.

\subsection{The Boolean algebra of clopen sets of $\partial T$}

Let $\mathrm{Bool}(\partial T)$ be the Boolean algebra of clopen subsets of $\partial T$. Let $\gamma\in \partial T$ and let us define the map $\varphi_\gamma:\mathrm{Bool}(\partial T)\to \mathbf{2}$ via
    \begin{equation*}
    \varphi_\gamma(A):=
        \begin{cases}
            1,&\text{if } \gamma\in A;\\
            0,&\text{if } \gamma\notin A.
        \end{cases}
    \end{equation*}

\begin{lemma}
\label{lemma: definition of varphi gamma}
    For any $\gamma\in \partial T$ the map $\varphi_\gamma:\mathrm{Bool}(\partial T)\to \mathbf{2}$ is a well-defined homomorphism of Boolean algebras.
\end{lemma}
\begin{proof}
    It is enough to prove that the set $I:=\{A\in \mathrm{Bool}(\partial T)\mid \gamma\notin A\}$ is a maximal ideal in $\mathrm{Bool}(\partial T)$ by \cref{lemma: ideal sufficient}. If $A,B\in I$, then $A\cup B\in I$ as $\gamma$ is neither contained in $A$ nor in $B$. Also if $A\in I$ and $B\in \mathrm{Bool}(\partial T)$ then $A\cap B\in I$ as $A$ does not contain $\gamma$. Lastly $\emptyset\in \mathrm{Bool}(\partial T)$ does not contain $\gamma$ so $I$ is non-empty. Thus the set $I$ is an ideal in $\mathrm{Bool}(\partial T)$. Now for any $A\in\mathrm{Bool}(\partial T)$ either $\gamma\in A$ or $\gamma\notin A$ but not both, so by \cref{lemma: maximal ideals contain either a or its complement} the ideal $I$ is maximal in $\mathrm{Bool}(\partial T)$.
\end{proof}

\begin{lemma}
\label{lemma: definition of inverse map}
    For any homomorphism $f:\mathrm{Bool}(\partial T)\to \mathbf{2}$ there exists a unique $\gamma\in \partial T$ such that $f=\varphi_\gamma$.
\end{lemma}
\begin{proof}
    Let us restrict the homomorphism $f$ to the subgraph $\{C_v\}_{v\in T}\subset\mathrm{Bool}(\partial T)$, which is isomorphic as a graph to $T$. By \textcolor{teal}{Lemmata} \ref{lemma: disjoint clopens} and \ref{lemma: non zero map} there exists a vertex $v\in T$ distinct from the root such that $f(C_v)=1$. Again by \textcolor{teal}{Lemmata} \ref{lemma: disjoint clopens} and \ref{lemma: non zero map} there exists a vertex $v_2\in T_{v_1}\setminus\{v_1\}$ such that $f(C_{v_2})=1$. Applying this argument inductively shows that the set $S:=\{C_v\mid f(C_v)=1\}$ is infinite. Furthermore since $f$ is a homomorphism, for any $C,D\in S$ we have
    $$f(C\cap D)=1.$$
    However, two cone sets are either disjoint or contained one in the other. Thus since $C\cap D=\emptyset$ would imply $f(C\cap D)=0$, we must have either $C\subseteq D$ or $D\subseteq C$ for any pair $C,D\in S$. This implies that the countable set $S$ is completely linearly ordered. Therefore $\gamma:=\bigcap_{C\in S} C$ is a well-defined uniquely determined element in the boundary $\partial T$. Note that $C_v\in S$ if and only if $\gamma\in C_v$.
    
    Let us conclude by showing $f=\varphi_\gamma$. Let $A\in \mathrm{Bool}(\partial T)$. If $A$ contains $\gamma$, then $A$ must contain a cone set $C\in S$ by \cref{lemma: disjoint clopens} and $f(A)=f(A\cup C)=f(A)\vee f(C)=1$. On the other hand, if $A$ does not contain $\gamma$, then $\neg A$ must contain~$\gamma$ and we get $f(A)=\neg f(\neg A)=\neg 1=0$. Thus $f=\varphi_\gamma$.
\end{proof}

\begin{theorem}
    The map 
    \begin{align*}
        F:\partial T&\to \mathrm{Hom}(\mathrm{Bool}(\partial T))\\
        \gamma&\mapsto \varphi_\gamma
    \end{align*}
    is a homeomorphism.
\end{theorem}
\begin{proof}
    By \textcolor{teal}{Lemmata} \ref{lemma: definition of varphi gamma} and \ref{lemma: definition of inverse map} the map $\gamma\mapsto \varphi_\gamma$ is 1-to-1 and invertible. Continuity follows from the definition for the topology of pointwise convergence of nets in $\mathrm{Hom}(\mathrm{Bool}(\partial T))$.
\end{proof}

\section{The structure lattice of a branch group}
\label{section: the structure lattice of a branch group}

In this section we introduce firstly the class of branch groups and secondly the structure lattice of a branch group. The remainder of the section is devoted to proving \cref{Theorem: canonical isomorphism of boolean algebras}.

\subsection{Branch groups}

Let $T$ be a spherically homogeneous rooted tree and $\mathrm{Aut}~T$ its group of graph automorphisms. Let us fix a subgroup $G\le \mathrm{Aut}~T$. We say $G$ is \textit{level-transitive} if its action on every level of $T$ is transitive. The group $G$ has a natural induced action on $\partial T$. We define the \textit{support} of $G$, denoted $\mathrm{Supp}(G)$, via
$$\mathrm{Supp}(G):=\{\gamma\in \partial T\mid \text{there exists }g\in G\text{ such that }\gamma^g\ne \gamma\}.$$

For a vertex $v\in T$, we write $\mathrm{st}_G(v)$ for the stabilizer of $v$ in $G$. We further define the \textit{rigid vertex stabilizer} $\mathrm{rist}_G(v)$ of $v$ as the subgroup of $\mathrm{st}_G(v)$ consisting of elements whose support is contained in $C_v$. We get from definition that 
$$[\mathrm{rist}_G(v),\mathrm{rist}_G(w)]=1$$
if $v$ and $w$ are incomparable vertices. Then for every $n\ge 1$, the direct product
$$\mathrm{Rist}_G(n):=\prod_{v\in \mathcal{L}_n}\mathrm{rist}_G(v)$$
is a well-defined subgroup of $G$, called the \textit{rigid level stabilizer} of the $n$th level.

We say that $G\le\mathrm{Aut}~T$ is a \textit{branch} group if $G$ is level-transitive and for every $n\ge 1$ the subgroup $\mathrm{Rist}_G(n)$ is of finite index in $G$. A \textit{branch action} of a group~$G$ on $T$ is simply a monomorphism $\rho:G\to \mathrm{Aut}~T$ such that $\rho(G)\le \mathrm{Aut}~T$ is a branch group.

\subsection{The structure lattice} The structure lattice of a just-infinite branch group was introduced by Wilson in \cite{Wilson71}. Here we present a slight modification which works for every branch group; compare \cite{AlejandraCSP, NewHorizonsWilson,WilsonBook,Wilson2}.

Let $G\le \mathrm{Aut}~T$ be a branch group. Let $L(G)$ be the collection of all subnormal subgroups of $G$ with finitely many conjugates, i.e. subnormal subgroups whose normalizers are of finite index in $G$. Then $L(G)$ can be endowed with a lattice structure defining for every $H,K\in L(G)$ 
$$H\wedge K:=H\cap K\in L(G) \quad \text{and}\quad  H\vee K:=\langle H,K\rangle\in L(G).$$

For $H,K\in L(G)$ we write $H\le_{\mathrm{va}} K$ if $H\le K$ and $H$ contains the commutator subgroup of a finite index subgroup of $K$. We define the equivalence relation $\sim $ in $L(G)$ as follows: for every $H,K\in L(G)$ we have
$$H\sim K\quad \text{if and only if}\quad H\cap K\le_{\mathrm{va}}H,K.$$
Furthermore, the equivalence relation $\sim$ is a congruence in $L(G)$, i.e. it is compatible with meets and joins. Therefore $L(G)/\sim$ is a well-defined lattice via
$$[H]\wedge [K]:=[H\cap K] \quad \text{and}\quad [H]\vee [K]:=[\langle H,K\rangle],$$
for every $H,K\in L(G)$. We shall write $\mathcal{L}(G):=L(G)/\sim$ and call this lattice the \textit{structure lattice of $G$}. The structure lattice $\mathcal{L}(G)$ satisfies the following properties:

\begin{proposition}[{see {\cite[Section 4]{Wilson71}} or {\cite[Section 3.1]{WilsonBook}}}]
    Let $G\le \mathrm{Aut}~T$ be a branch group. Then the structure lattice  $\mathcal{L}(G)$ satisfies:
    \begin{enumerate}[\normalfont(i)]
    \item it admits the distinguished elements $0:=[\{1\}]$ and $1:=[G]$;
    \item it is distributive;
    \item it is uniquely complemented.
    \end{enumerate}
    Thus $\mathcal{L}(G)$ is a Boolean algebra.
\end{proposition}

Finally, observe that for $g\in G$ and $H\in L(G)$, we have $[H]^g=[H^g]$. Thus there is a well-defined action of $G$ on $\mathcal{L}(G)$ induced by the conjugation action of $G$ on its subgroups.

\subsection{Proof of \cref{Theorem: canonical isomorphism of boolean algebras}}
We conclude the section by proving \cref{Theorem: canonical isomorphism of boolean algebras}. We shall fix a branch action $\rho:G\to \mathrm{Aut}~T$ and write $G\le \mathrm{Aut}~T$ for the remainder of the section. First we show that the support of a subgroup with finite index normalizer is clopen. For that we shall need the following lemma in \cite{Leemann}:

\begin{lemma}[{see {\cite[Lemma 2.5]{Leemann}}}]
    \label{lemma: leemann rigid}
    Let $G\le \mathrm{Aut}~T$ be a branch group. Then for every $n\ge 1$, there exists $N_n\ge 1$ such that $\mathrm{Rist}_G(n)$ acts level-transitively on the subtrees rooted at level $n+N_n$.
\end{lemma}

\begin{proposition}
\label{proposition: support of subnormal is clopen}
    Let $G\le \mathrm{Aut}~T$ be a branch group and let $H\le G$ be a subgroup with finitely many conjugates in $G$. Then $\mathrm{Supp}(H)$ is clopen.
\end{proposition}

\begin{proof}
    For any vertex $v\in T$ either $v$ is fixed by $H$ or $C_v\subseteq \mathrm{Supp}(H)$. Thus since $G\le \mathrm{Aut}~T$, we can check which vertices are moved and which ones are fixed by $H$ at each level of $T$ and hence decompose $\mathrm{Supp}(H)$ as the disjoint countable union
    \begin{align}
    \label{align: support decomposition}
        \mathrm{Supp}(H)=\bigsqcup_{v\in V}C_v
    \end{align}
    for some subset $V\subset T$ of pairwise incomparable vertices all moved by $H$. Note that in particular \cref{align: support decomposition} shows that $\mathrm{Supp}(H)$ is always open in $\partial T$. It remains to prove that $\mathrm{Supp}(H)$ is closed. We claim that the subset $V\subseteq T$ in \cref{align: support decomposition} is finite. This implies that  $\mathrm{Supp}(H)$ is a finite union of clopen sets and thus it is clopen itself. Then let us prove our claim. 
    
    Assume by contradiction that $V$ is infinite. We shall assume that $V$ is ordered by graded lexicographical order. Furthermore we may assume that $V$ does not become finite if for any subset $W\subseteq V$ such that 
    \begin{align}
        \label{align: no redundancies}
        \bigsqcup_{w\in W}C_w=C_v
    \end{align}
    for some $v\in T$, one replaces $W$ with $v$ in $V$. Note that for any $g\in G$ we have
    \begin{align}
    \label{align: support decomposition for conjugates}
        \mathrm{Supp}(H^g)=\bigsqcup_{v\in V}C_{v^g}=\bigsqcup_{w\in V^g}C_{w},
    \end{align}
    where $V^g$ consists of infinitely many pairwise incomparable vertices all moved by~$H^g$.
    
    We shall construct infinitely many distinct conjugates of $H$ which yields a contradiction. For that, it is enough to find an infinite sequence of levels $\{\ell_n\}_{n\ge 1}$ and of elements $\{g_n\}_{n\ge 1}\subseteq G$ such that for every $n\ge 1$ the subgroups $H^{g_{n+1}}$ and $H^{g_n}$ have the same action on level~$\ell_n$ but they move different vertices at level $\ell_{n+1}$. We fix $\ell_1=1$ and $g_1=1$. We construct these two sequences by induction on $n\ge 1$. Let $N_n\ge 1$ be such that $\mathrm{Rist}_G(\ell_n)$ acts level-transitively on the subtrees rooted at level $\ell_n+N_n$, which is well-defined by \cref{lemma: leemann rigid}. Now, there is $w_n\in V^{g_n}$, i.e. $w_n$ is moved by $H^{g_n}$, at some level $\ell_{n+1}$ such that:
    \begin{enumerate}[\normalfont(i)]
        \item $\ell_{n+1}>\ell_n+N_n;$
        \item if $v_n$ denotes the unique vertex at level $\ell_n+N_n$ above $w_n$ then there is a descendant $\widetilde{w}_n$ of $v_n$ at level $l_{n+1}$ fixed by $H^{g_n}$.
    \end{enumerate}
    Note that condition (i) is guaranteed by $V^{g_n}$ being infinite while condition (ii) is guaranteed by \cref{align: support decomposition for conjugates} and the assumption in \cref{align: no redundancies}. Let $u_n$ be the unique vertex at level $\ell_n$ above both $w_n$ and $v_n$. Consider $h_n\in \mathrm{rist}_G(u_n)$ such that $w_n^{h_n}=\widetilde{w}_n$, which exists by the level-transitivity of $\mathrm{rist}_G(u_n)$ on the subtree rooted at $v_n$. Let $g_{n+1}:=g_nh_n$. Then $\widetilde{w}_n$ is moved by $H^{g_{n+1}}$ but it is fixed by $H^{g_n}$. Indeed, there exists $h\in H$ such that $w_n^{h^{g_n}}\ne w_n$, and thus
    $$\widetilde{w}_n^{h^{g_{n+1}}}=w_n^{h_nh_n^{-1}g_n^{-1}hg_nh_n}=w_n^{g_n^{-1}hg_nh_n}\ne w_n^{h_n}=\widetilde{w}_n.$$
    However, both subgroups $H^{g_{n+1}}$ and $H^{g_n}$ have the same action on level $\ell_{n}$, concluding the proof.
\end{proof}

Note that the proof above shows that the support of an arbitrary subgroup $H\le \mathrm{Aut}~T$ is open.

We need the following lemma in \cite{Dominik}, which is stated more generally for weakly branch groups, where a \textit{weakly branch} group $G$ is a level-transitive subgroup of $\mathrm{Aut}~T$ such that $\mathrm{Rist}_G(n)\ne 1$ for every $n\ge 1$.

\begin{lemma}[{see {\cite[Lemma 2.16]{Dominik}}}]
\label{lemma: subnormal contains rist'..'}
    Let $G\le \mathrm{Aut}~T$ be a weakly branch group and~$H$ a $k$-subnormal subgroup of $G$. Then $H\ge \mathrm{rist}_G(v)^{(k)}$ for any $v\in T$ moved by $H$.
\end{lemma}

Now we can give an explicit description of the structure lattice of a branch group:

\begin{proposition}
\label{proposition: representation of subnormal}
    Let $G\le \mathrm{Aut}~T$ a branch group and let $H\in L(G)$. Then 
    $$[H]=\Big[\prod_{v\in \mathrm{Supp}(H)} \mathrm{rist}_G(v)\Big]$$
    in the structure lattice $\mathcal{L}(G)$.
\end{proposition}
\begin{proof}
    If $H=\{1\}$ then $\mathrm{Supp}(H)=\emptyset$ and the result is clear. Thus, let $H\le G$ be a non-trivial $k$-subnormal subgroup with finitely many conjugates. Then there exists a vertex $v\in T$ moved by $H$ and by \cref{lemma: subnormal contains rist'..'} we have
    $$H\ge \mathrm{rist}_G(v)^{(k)}.$$
    Now we show that $\mathrm{rist}_G(v)^{(k)}$ is subnormal in $G$ and that it has the same number of distinct conjugates as $H$. First $\mathrm{rist}_G(v)$ is subnormal in $G$ as it is normal in the corresponding rigid level stabilizer, which is itself normal in $G$. Thus $\mathrm{rist}_G(v)^{(k)}$ is subnormal in $G$ too as it is itself normal in $\mathrm{rist}_G(v)$. We know that $N_G(\mathrm{rist}_G(v))=\mathrm{st}_G(v)$, which is of finite index in $G$. Now, since $\mathrm{rist}_G(v)^{(k)}$ is characteristic in $\mathrm{rist}_G(v)$, it is normal in $\mathrm{st}_G(v)$ and therefore
    $$N_G(\mathrm{rist}_G(v)^{(k)})=\mathrm{st}_G(v)$$
    as $\mathrm{Supp}(\mathrm{rist}_G(v)^{(k)})$ is contained in $C_v$. Thus $\mathrm{rist}_G(v)^{(k)}$ has the same number of distinct conjugates as $\mathrm{rist}_G(v)$, namely the number of vertices in the $G$-orbit of $v$ 
    by the orbit-stabilizer theorem.  By the level-transitivity of $G$ this is precisely the number of vertices at the same level as $v$.

    Now by \cref{proposition: support of subnormal is clopen}, the support of $H$ is clopen and a disjoint union of cone sets. Let $\mathrm{Supp}(H)=\bigsqcup_{i=1}^nC_{v_i}$, where we assume as before that each $v_i$ is moved by $H$. Then the above reasoning applies to each $v_i$ and we get
    $$H\ge \prod_{i=1}^{n}\mathrm{rist}_G(v_i)^{(k)}.$$
    We may assume $v_1,\dotsc,v_n$ are all at the same level by replacing if necessary a vertex with all its descendants at a lower level. Then 
    \begin{align}
    \label{align: h contains some direct product of commutators of rists}
        H\ge \prod_{i=1}^{n}\mathrm{rist}_G(v_i)^{(k)}= \big(\prod_{i=1}^n\mathrm{rist}_G(v_i)\big)^{(k)}
    \end{align}
    as the derived subgroup of a direct product is no more than the direct product of the derived subgroups of the factors. Now 
    $$\big(\prod_{i=1}^n\mathrm{rist}_G(v_i)\big)^{(k)}\le_{\mathrm{va}}\big(\prod_{i=1}^n\mathrm{rist}_G(v_i)\big)^{(k-1)}\le_{\mathrm{va}}\dotsb\le_{\mathrm{va}}\prod_{i=1}^n\mathrm{rist}_G(v_i)$$
    and thus
    \begin{align}
        \label{align: equality in structure lattice}
        \Big[\big(\prod_{i=1}^n\mathrm{rist}_G(v_i)\big)^{(k)}\Big]=\Big[\big(\prod_{i=1}^n\mathrm{rist}_G(v_i)\big)^{(k-1)}\Big]=\dotsb=\Big[\prod_{i=1}^n\mathrm{rist}_G(v_i)\Big]
    \end{align}
    in the structure lattice. Hence by \textcolor{teal}{Equations (}\ref{align: h contains some direct product of commutators of rists}\textcolor{teal}{)} and \textcolor{teal}{(}\ref{align: equality in structure lattice}\textcolor{teal}{)} we obtain
    $$H\cap \prod_{i=1}^n\mathrm{rist}_G(v_i)\le_{\mathrm{va}}\prod_{i=1}^n\mathrm{rist}_G(v_i).$$
    Now we have
    $$\prod_{i=1}^n\mathrm{rist}_G(v_i)\le_f \mathrm{Stab}_G(\partial T\setminus \mathrm{Supp}(H))\le G$$
    as $G$ is branch and
    $$\prod_{i=1}^n\mathrm{rist}_G(v_i)=\mathrm{Rist}_G(N)\cap \mathrm{Stab}_G(\partial T\setminus \mathrm{Supp}(H)),$$
    where $N$ is the common level of $T$ at which all the vertices $v_1,\dotsc,v_n$ lie. Since $H\le \mathrm{Stab}_G(\partial T\setminus \mathrm{Supp}(H))$, we also get 
    $$H\cap \prod_{i=1}^n\mathrm{rist}_G(v_i)\le_f H.$$
    Therefore 
    $$ [H]=\Big[\prod_{i=1}^n\mathrm{rist}_G(v_i)\Big]=\Big[\prod_{v\in \mathrm{Supp}(H)}\mathrm{rist}_G(v)\Big]$$
    in the structure lattice, where the second equality follows from 
    \begin{align*}
        \prod_{i=1}^n\mathrm{rist}_G(v_i)=\mathrm{Rist}_G(N)\cap \prod_{v\in \mathrm{Supp}(H)}\mathrm{rist}_G(v) &\le_f\prod_{v\in \mathrm{Supp}(H)}\mathrm{rist}_G(v).\qedhere
    \end{align*}
\end{proof}

\begin{corollary}
    \label{corollary: support is well-defined in the lattice}
    Let $G\le \mathrm{Aut}~T$ be a branch group and let $H,K\in \mathcal{L}(G)$ be such $H\sim K$. Then $\mathrm{Supp}(H)=\mathrm{Supp}(K)$.
\end{corollary}
\begin{proof}
    By \cref{proposition: representation of subnormal}
    $$\Big[\prod_{v\in \mathrm{Supp}(H)} \mathrm{rist}_G(v)\Big]=[H]=[K]=\Big[\prod_{v\in \mathrm{Supp}(K)} \mathrm{rist}_G(v)\Big].$$
    Now if $\mathrm{Supp}(H)\ne \mathrm{Supp}(K)$ we may assume without loss of generality that there exists $v\in \mathrm{Supp}(H)$ such that $C_v\cap \mathrm{Supp}(K)=\emptyset$. Then

    $$\Big(\prod_{w\in \mathrm{Supp}(H)} \mathrm{rist}_G(w)\Big)\cap \Big(\prod_{w\in \mathrm{Supp}(K)} \mathrm{rist}_G(w)\Big)\not\le_{\mathrm{va}} \prod_{w\in \mathrm{Supp}(H)} \mathrm{rist}_G(w)$$
    as $\mathrm{rist}_G(v)$ is not virtually abelian. This contradicts $H\sim K$.
\end{proof}

\begin{proof}[Proof of \cref{Theorem: canonical isomorphism of boolean algebras}]
    First $\Phi$ is well-defined by \cref{proposition: support of subnormal is clopen} and \cref{corollary: support is well-defined in the lattice}. \cref{proposition: representation of subnormal} yields that $\Phi$ is both surjective and injective. Finally $G$-equivariance of $\Phi$ follows from \cref{align: support decomposition for conjugates}.
\end{proof}



\bibliographystyle{unsrt}

\end{document}